\documentclass{amsart}
\usepackage{amsrefs}
\usepackage{enumerate}
\usepackage{amsmath,amssymb,amsfonts}
\usepackage{amscd, amssymb, latexsym, amsmath, amscd}
\usepackage[all]{xy}
\SelectTips{cm}{12}

\usepackage[sc]{mathpazo} 
\DeclareFontFamily{U}{wncy}{}
    \DeclareFontShape{U}{wncy}{m}{n}{<->wncyr10}{}
    \DeclareSymbolFont{mcy}{U}{wncy}{m}{n}
    \DeclareMathSymbol{\Sh}{\mathord}{mcy}{"58} 

\newcommand{\map}[3]{#1 \colon #2 \longrightarrow #3}
\newcommand{\abmap}[2]{#1 \longrightarrow #2}
\newcommand{\Q}{\mathbf{Q}}
\newcommand{\C}{\mathbf{C}}
\newcommand{\ZZ}{\mathbf{Z}}
\newcommand{\RRf}{\mathcal{R}_{\mathrm{fin}}}
\newcommand{\LL}{\mathcal{L}}
\newcommand{\ff}{\mathfrak{f}}
\newcommand{\mult}{^\times}
\newcommand{\qs}{{\textup{qs}}}

\newcommand{\Special}{{\textup{S}}}
\newcommand{\grouphom}{\lambda} % symbol for hom from G_1 to G_2
\newcommand{\simil}{\mu} % symbol for the similitude character
\DeclareMathOperator{\disc}{disc}
\DeclareMathOperator{\Gal}{Gal}
\DeclareMathOperator{\Hom}{Hom}
\DeclareMathOperator{\Span}{span}
\DeclareMathOperator{\GL}{GL}
\DeclareMathOperator{\SL}{SL}
\DeclareMathOperator{\GU}{GU}

\DeclareMathOperator{\SU}{SU}
\DeclareMathOperator{\GSp}{GSp}
\DeclareMathOperator{\Sp}{Sp}
\DeclareMathOperator{\SO}{SO}
\DeclareMathOperator{\OO}{O}
\DeclareMathOperator{\GO}{GO}
\DeclareMathOperator{\GSO}{GSO}
\DeclareMathOperator{\U}{U}

\DeclareMathOperator{\Image}{Im}
\DeclareMathOperator{\infl}{infl}
\DeclareMathOperator{\mm}{m}
\DeclareMathOperator{\Res}{Res}
\newcommand{\lsup}[1]{{}^{#1}}
\newcommand{\inv}{^{-1}}

\newcommand{\sfT}{\mathsf{T}}
\DeclareMathOperator{\cInd}{c-Ind}
\DeclareMathOperator{\Ind}{Ind}
\newcommand{\G}{{\mathrm G}}
\newcommand{\Z}{{\mathrm Z}}

\newtheorem{thm}{Theorem}

\newtheorem{lem}[thm]{Lemma}
\newtheorem{cor}[thm]{Corollary}
\newtheorem{conj}[thm]{Conjecture}

\theoremstyle{definition}

\newtheorem{rem}[thm]{Remark}

\title{Multiplicity upon restriction to the derived subgroup}

\author{Jeffrey D. Adler}
\address{
Department of Mathematics and Statistics \\
American University \\
4400 Massachusetts Ave NW \\
Washington, DC  20016-8050 \\
USA
}
\email{jadler@american.edu}
\author{Dipendra Prasad}
\address{
School of Mathematics \\
Indian Institute of Technology \\
Mumbai \\
India
}
\address{
Tata Institute of Fundamental Research \\
Mumbai  \\
India 
}
\address{
Laboratory of Modern Algebra and Applications \\
Saint-Petersburg State University \\
Russia
}
\email{prasad.dipendra@gmail.com}

\begin{document}
\begin{abstract}
We present a conjecture
on multiplicity of irreducible representations of a subgroup $H$
contained in the irreducible representations of a group $G$,
with $G$ and $H$ having the same derived groups.
We point out some consequences of the conjecture,
and verification of some of the consequences.
We give an explicit example of multiplicity $2$ upon restriction,
as well as certain theorems in the context of classical groups
where the multiplicity is $1$.
\end{abstract}
\maketitle

\section{Introduction}
Suppose $k$ is a local field,
$\G$ is a connected reductive $k$-group,
$\G'$ is a subgroup of $\G$ containing the derived group,
and $\pi$ is a smooth, irreducible, complex representation 
of $\G(k)$.
In an earlier work
\cite{adler-dprasad:mult-one},
the authors
showed that for many choices of $\G$,
the restriction $\Res^{\G(k)}_{\G'(k)}\pi$ decomposes without multiplicity.

A number of years ago, in the process of identifying situations
where multiplicity one did not  hold,
one of us discovered an example of a depth-zero supercuspidal representation of $\GU(2d,2d)$,
a $k$-quasisplit group,
whose restriction to $\SU(2d,2d)$ decomposes with
multiplicity two,
and the other formulated a conjecture in the form
of a reciprocity law involving
enhanced Langlands parameters.
In this paper, we present both the example and the conjecture, 
together with some consequences of the latter,
and a verification of some of those consequences.
Besides these, the paper proves several results by elementary means
involving classical groups where multiplicity one holds. 

A complete analysis of decomposition of the unitary principal
series for $\U(n,n)$ and its restriction to $\SU(n,n)$
was done by Keys
\cite{keys:unitary-even},
who also phrased his results in terms of ``reciprocity''
theorems for $R$-groups;
in particular, he found cases of multiplicity greater than one.

After presenting our conjecture (\S\ref{sec:conj}),
we give some of the heuristics
behind it.
In the formulation of the conjecture, we
have considered a more general situation than that of a subgroup.
We consider $\G_1$ and $\G_2$ to
be two connected reductive groups over a local field $k$,
and $\grouphom: \G_1\rightarrow \G_2$ a $k$-homomorphism
that is a central isogeny when restricted to their derived subgroups,
allowing us to ``restrict'' representations
of $\G_2(k)$ to $\G_1(k)$.
Since under such a homomorphism $\grouphom$,
the image of $\G_1(k)$ is a normal subgroup of $\G_2(k)$
with abelian quotient,
all the irreducible representations of $\G_1(k)$
which appear in this restriction problem for a given irreducible representation
of $\G_2(k)$ appear with the same multiplicity.
In \S\ref{sec:conj-rem},
we verify that for our conjectural multiplicity, this relationship does indeed hold.
We show (\S\ref{sec:tempered})
that if the conjecture is true for tempered representations,
then via the Langlands classification it holds for all representations.

Our conjecture (for $\grouphom: \G_1\rightarrow \G_2$ a $k$-homomorphism),
implies multiplicity one in situations
where Langlands parameters for $\G_1$ have abelian component groups.
We list a few such situations in \S\ref{sec:consequences}, and prove
multiplicity one for restriction
from $\GU(n)$ to $\U(n)$
(\S\ref{sec:unitary}).
Along the way, we prove multiplicity one in some other cases
where it follows from elementary considerations.
In \S\ref{sec:example-mult}, we present an
example of a depth-zero supercuspidal representation
of quasi-split $\GU(2d,2d)$ that decomposes with multiplicity two
upon restriction to $\SU(2d,2d)$.
Finally (\S\ref{sec:generalities}),
we give a general procedure for constructing higher 
multiplicities.

\section{The Conjecture on multiplicities}
\label{sec:conj}

Let $\G^{\qs}_1$ and $\G^{\qs}_2$ be two connected quasi-split reductive groups over a local field $k$
and let $\grouphom: \G^{\qs}_1\rightarrow \G^{\qs}_2$ be a $k$-homomorphism
that is a central isogeny when restricted to their derived subgroups.
In what follows we will be twisting $\G^{\qs}_1$ by a cohomology class in
$H^1(\Gal(\bar{k}/k),\G^{\qs}_1(\bar{k}))$ to construct a pure inner form $\G_1$ of
$\G^{\qs}_1$;
simultaneously,
by twisting $\G^{\qs}_2$ by the image of this class under the map
$H^1(\Gal(\bar{k}/k),\G^{\qs}_1(\bar{k})) \rightarrow 
H^1(\Gal(\bar{k}/k),\G^{\qs}_2(\bar{k}))$,
we will have a pure inner form $\G_2$ of
$\G^{\qs}_2$,
together with map of algebraic groups that we will still call
$\grouphom: \G_1\rightarrow \G_2$ which will appear
in considerations below, all coming from an element of
$H^1(\Gal(\bar{k}/k),\G^{\qs}_1(\bar{k}))$. 

The map $\grouphom: \G_1\rightarrow \G_2$ 
gives rise to a ``restriction'' map
from representations of $\G_2(k)$ to those of $\G_1(k)$,
and 
from Silberger \cite{silberger:isogeny-restriction}
one knows that the restriction of an irreducible
representation of $\G_2(k)$ is a finite direct sum of irreducible
representations of $\G_1(k)$.
In particular, we obtain a functor
$\grouphom^\star : \RRf(\G_2(k)) \rightarrow \RRf(\G_1(k))$,
where $\RRf(H)$ denotes the
category of smooth, finite-length representations of a group $H$.

Let ${}^L \G_1 = {\widehat {\G}}_1\rtimes W'_k$ and $ {}^L\G_2
= {\widehat {\G}}_2\rtimes W'_k$ be the $L$-groups associated to 
the quasi-split reductive groups $\G^{\qs}_1$ and $\G^{\qs}_2$ respectively.
The map $\grouphom: \G^{\qs}_1\rightarrow \G^{\qs}_2$ also gives rise to a homomorphism of $L$-groups:
$$
{}^L\grouphom: {}^L \G_2 \longrightarrow {}^L\G_1,
$$
as well as a homomorphism of their centers:
$$
{}^L \grouphom
:
Z({\widehat{\G}_2})^{W_k}\longrightarrow Z({{\widehat {\G}}_1})^{W_k}.
$$
In particular,
a character $\chi_1$ of $\pi_0(Z({{\widehat {\G}}_1})^{W_k})$
gives rise to a character $\chi_2$ of 
$\pi_0(Z({{\widehat {\G}}_2})^{W_k})$
which by the Kottwitz isomorphism
(assuming $k$ to be nonarchimedean at this point):
$$
H^1(\Gal(\bar{k}/k),\G^{\qs}_i(\bar{k}))
\cong
\Hom (\pi_0(Z({{\widehat\G}_i})^{W_k}), \Q/\Z),
$$
constructs pure inner forms $\G_1$ of $\G_1^{\qs}$ and $\G_2$ of $\G^{\qs}_2$,
together with a map $\grouphom: \G_1\rightarrow \G_2$ as before.

Let $\varphi_2: W'_k \rightarrow {}^L\G_2$,
and
$\varphi_1= {}^L\grouphom \circ \varphi_2 : W'_k \rightarrow {}^L\G_1$
be associated Langlands parameters,
where $W'_k = W_k\times \SL_2(\C)$, with $W_k$ the Weil group of $k$.
Then ${}^L\grouphom$ gives rise to a homomorphism
%$$
%{\widehat {\grouphom}}:
%Z_{\widehat{\G}_2}(\varphi_2) \longrightarrow Z_{{\widehat {\G}}_1}(\varphi_1)
%$$
of centralizers of the images of the  parameters 
$\varphi_1$ with values in ${}^L\G_1$ and $\varphi_2$ with values in ${}^L\G_2$,
and also a homomorphism of the groups of connected components
of their centralizers:
$$
\pi_0({}^L \grouphom)
:
\pi_0(Z_{\widehat{\G}_2}(\varphi_2))
\longrightarrow
\pi_0(Z_{{\widehat {\G}}_1}(\varphi_1)).
$$
This allows one to `restrict' representations of
$\pi_0(Z_{{\widehat\G}_1}(\varphi_1))$
to representations of
$\pi_0(Z_{{\widehat\G}_2}(\varphi_2))$, giving rise to  the restriction functor 
$$
\grouphom_\star:
K_0(\pi_0(Z_{{\widehat\G}_1}(\varphi_1)))
\rightarrow
K_0(\pi_0(Z_{{\widehat\G}_2}(\varphi_2))),
$$
where $K_0(H)$ denotes the Grothendieck group of finite-length
representations of a group $H$.

The formulation of our conjecture below presumes that
the local Langlands correspondence involving enhanced Langlands parameters
has been achieved,
giving rise to a bijection between enhanced Langlands parameters 
and the set of isomorphism classes of irreducible admissible representations of 
all pure inner forms of quasi-split groups.
This will be needed for \emph{both} of the groups $\G_1$ and $\G_2$;
it is possible on the other hand
that one could reverse this role,
and use the conjectural multiplicity formula to construct
an enhanced Langlands parametrization for $\G_2$, knowing it for $\G_1$.

\begin{conj} 
\label{conj:main}
\begin{enumerate}[(a)]
\item
\label{item:conj1}
Let $\G_1$ and $\G_2$ be two connected reductive groups over a local field $k$
and $\grouphom: \G_1\rightarrow \G_2$, a $k$-homomorphism
that is a central isogeny when restricted to their derived subgroups.
For $i=1,2$,
let
$\pi_i$
be an irreducible admissible representation of $\G_i(k)$
with Langlands parameter $\varphi_i$. Let 
$$
\mm(\pi_2,\pi_1)
:= \dim \Hom_{\G_1(k)}[\pi_1, \grouphom^\star \pi_2]
=  \dim \Hom_{\G_1(k)}[\grouphom^\star \pi_2, \pi_1].
$$
Then 
$\mm(\pi_2,\pi_1) = 0$ unless
$\varphi_1 = {}^L \grouphom \circ \varphi_2$.
\item
Let $\G^{\qs}_1$ and $\G^{\qs}_2$ 
be two connected reductive quasi-split groups over a local field $k$
and $\grouphom: \G^{\qs}_1\rightarrow \G^{\qs}_2$, a $k$-homomorphism
that is a central isogeny when restricted to their derived subgroups.
Let $\varphi_1$ and $\varphi_2 $ be
Langlands parameters associated to the groups  
$\G^{\qs}_1$ and $\G^{\qs}_2$ with  
$\varphi_1 
= {}^L \grouphom \circ \varphi_2$, and let $\chi_i$
be characters of their component groups
$\pi_0(Z_{\widehat{\G}_i}(\varphi_i))$.
Then, if $\Hom_{\pi_0(Z(\varphi_2))}[ \chi_2, \grouphom_\star \chi_1]$ 
is nonzero, the characters $\chi_i$ define
pure inner forms $\G_i$ of  $\G^{\qs}_i$ together with a $k$-homomorphism, $\grouphom: \G_1\rightarrow \G_2$, 
as discussed earlier.
Then if $\pi_i = \pi(\varphi_i,\chi_i)$
are the corresponding irreducible admissible 
representations of $\G_i(k)$, we have
$$
\mm(\pi_2,\pi_1)
= \dim \Hom_{\pi_0(Z(\varphi_2))}[ \chi_2, \grouphom_\star \chi_1].
$$
\end{enumerate}
\end{conj}

The main heuristic for the conjectural multiplicity is the following.
\begin{enumerate}
\item
For any $L$-packet $\{\pi \}$ on any reductive group $\G(k)$
defined by a parameter $\varphi$,
thus $\{\pi\} = \{ \pi_{(\varphi, \chi)}\}$ where one takes those characters $\chi$ of the component group 
which have a particular restriction to $Z(\widehat{\G})^{W_k}$ defining the group $\G(k)$ assumed to be a  pure inner form
of a fixed quasi-split group $\G^{\qs}$,
$$
\sum _{\chi} \chi(1) \Theta(\pi_{(\varphi, \chi)})
$$
is a stable distribution on $\G(k)$.
Here, for any admissible representation $\pi$
we are letting $\Theta(\pi)$ denote
its character, regarded as a distribution on $\G(k)$.

\item
For a homomorphism $\grouphom: \G_1\rightarrow \G_2$ of reductive groups over $k$ 
which is an isogeny when restricted to their derived subgroups,
the pullback of a stable distribution on $\G_2(k)$ is a stable distribution 
on $\G_1(k)$.

\item
The restriction to $\G_1(k)$
of an irreducible
representation $\pi_2$ of $\G_2(k)$
is a finite-length (completely reducible)
representation of $\G_1(k)$,
whose irreducible components are all in the same $L$-packet.
This $L$-packet for $\G_1(k)$ depends only on the $L$-packet
for $\G_2(k)$ containing $\pi_2$.
If the Langlands parameter of our $L$-packet for $\G_2(k)$ is
$\varphi_2: W'_k \rightarrow {}^L\G_2$,
then the 
Langlands parameter of our $L$-packet for $\G_1(k)$ is
$\varphi_1:= {}^L\grouphom \circ \varphi_2 : W'_k \rightarrow {}^L\G_1$.
(This is part (\ref{item:conj1}) of the conjecture.)

\item
If Conjecture \ref{conj:main} is true,
then the pullback from $\G_2(k)$ to $\G_1(k)$
of the distribution
$$
\sum _{\chi_2} \chi_2(1) \Theta(\pi_{(\varphi_2, \chi_2)})
$$
where the sum is taken over those characters $\chi_2$ of the component group 
which have a particular restriction to $Z(\widehat{G}_2)^{W_k}$ 
defining the group $\G_2(k)$ assumed to be a 
pure inner form
of a fixed quasi-split group $\G_2^{\qs}(k)$, is a stable distribution on $\G_1(k)$ as we check now.

By Conjecture \ref{conj:main},
the pullback of the distribution
$\Theta_{\pi_2} = \Theta(\pi_{(\varphi_2, \chi_2)})$
on $\G_2(k)$ to $\G_1(k)$ is
$$
\sum_{\pi_1} m(\pi_2,\pi_1) \Theta(\pi_1)
=
\sum_{\chi_1}\Theta(\pi_{(\varphi_1, \chi_1)})
	\dim \Hom_{\pi_0(Z(\varphi_2))}[\chi_2, \grouphom_\star \chi_1].
$$
Therefore, the pullback to $\G_1(k)$ of the distribution
$\sum _{\chi_2} \chi_2(1) \Theta(\pi_{(\varphi_2, \chi_2)})$
on $\G_2(k)$
is (assuming Conjecture \ref{conj:main})
$$
\sum _{\chi_1, \chi_2}
	\chi_2(1) \Theta(\pi_{(\varphi_1, \chi_1)})
	\dim \Hom_{\pi_0(Z(\varphi_2))}[\chi_2, \grouphom_\star \chi_1],
$$
        which 
        is the same as
$$
\sum _{\chi_1,\chi_2}
	\Theta(\pi_{(\varphi_1, \chi_1)})
	\dim \Hom_{\pi_0(Z(\varphi_2))}[ \chi_2(1) \chi_2, \grouphom_\star \chi_1],
$$
where the sum is taken over all pairs of characters $\chi_1,\chi_2$ with particular restrictions to
$Z(\widehat{\G}_1)^{W_k}$ 
and $Z(\widehat{\G}_2)^{W_k}$.
Observe that
those characters $\chi_2$ whose restrictions to $Z(\widehat{\G}_2)^{W_k}$
are not compatible with the restriction of $\chi_1$ to $Z(\widehat{\G}_1)^{W_k}$
contribute $0$ to the sum.
Therefore, we can take the sum over all $\chi_2$.  The sum then is the same as
\begin{equation}
\tag{$*$}
\sum _{\chi_1}
	\Theta(\pi_{(\varphi_1, \chi_1)})
	\dim \Hom_{\pi_0(Z(\varphi_2))}[R, \grouphom_\star \chi_1],
\end{equation}
where $R = \sum\chi_2(1) \chi_2$
is the regular representation of $\pi_0(Z(\varphi_2))$.

By Schur orthogonality,
$$
\dim\Hom_{\pi_0(Z(\varphi_2))} [\chi_2, \grouphom_\star \chi_1]
=
\frac{1}{|\pi_0(Z(\varphi_2))|}
\sum_{g \in \pi_0(Z(\varphi_2)) }\chi_1(\grouphom^\star g) \bar{\chi}_2(g),
$$
where $\grouphom^\star$ denotes the map
$\map{\pi_0( \lsup L \grouphom)}{\pi_0(Z(\varphi_2))}{\pi_0(Z(\varphi_1))}$.
So
\begin{align*} 
  \dim\Hom_{\pi_0(Z(\varphi_2))}
          [R, \grouphom_\star \chi_1 ]
& = \frac{1}{|\pi_0(Z(\varphi_2))|}
	\sum_{g\in \pi_0(Z(\varphi_2))}\chi_1( \grouphom^\star g) \chi_R(g),
\end{align*}
where $R$ is the regular representation of $\pi_0(Z(\varphi_2))$
and $\chi_R$ its character,
thus
$$
\chi_R(g) =
\begin{cases}
0 & \text{if $g$ is not the identity}, \\
|\pi_0(Z(\varphi_2))| & \text{if $g$ is the identity}.
\end{cases}
$$
Therefore,
$$
\dim\Hom_{\pi_0(Z(\varphi_2))}[ R, \grouphom_\star \chi_1 ]
=
\chi_1(1).
$$
By $(*)$
it follows that the pullback of the distribution
$\sum _{\chi_2} \chi_2(1) \Theta(\pi_{(\varphi_2, \chi_2)})$ on $\G_2(k)$
to $\G_1(k)$ is equal to 
$\sum _{\chi_1} \chi_1(1) \Theta(\pi_{(\varphi_1, \chi_1)})$,
where the sum is taken over those $\chi_1$ 
with a given restriction to $Z(\widehat{\G}_1)^{W_k}$.
Thus the  pullback of the distribution
$\sum _{\chi_2} \chi_2(1) \Theta(\pi_{(\varphi_2, \chi_2)})$ on $\G_2(k)$
to $\G_1(k)$  is a stable distribution on $\G_1(k)$ which is what we set out to prove.
\end{enumerate}

\begin{rem}
A weaker version of our conjecture says that the pullback to $\G_1(k)$
of the stable character
$\sum_\chi \chi(1)\Theta_\chi$ on $\G_2(k)$
is
$\sum_\mu \mu(1)\Theta_\mu$ on $\G_1(k)$,
where both of the sums are over the characters of component groups
defining fixed pure inner forms that are $\G_2$ and $\G_1$, respectively.
\end{rem}

\section{Some remarks on the multiplicity formula}
\label{sec:conj-rem}
Conjecture \ref{conj:main}
relating $\mm(\pi_2,\pi_1)$
with $\dim \Hom_{\pi_0(Z(\varphi_2))}[ \grouphom_\star \chi_1, \chi_2]$
can be considered as a set of assertions keeping $\pi_2$ fixed
and varying $\pi_1$,
or keeping $\pi_1$ fixed and varying $\pi_2$, say inside an
$L$-packet for $\G_2(k)$.
It is easy to see that for
$\G_1$ and $\G_2$ two reductive groups over a local field $k$,
and $\grouphom: \G_1\rightarrow \G_2$, a $k$-homomorphism
that is a central isogeny when restricted to their derived subgroups,
the image of $\G_1(k)$ inside $\G_2(k)$
is a normal subgroup,
and therefore every irreducible representation of $\G_1(k)$
that appears inside a given
irreducible representation $\pi_2$
of $\G_2(k)$ does so with the same multiplicity
(depending of course on $\pi_2$).
This section aims to prove this as a
consequence of our Conjecture \ref{conj:main}.

This section is meant to prove that
$\dim \Hom_{\pi_0(Z(\varphi_2))}[ \grouphom_\star \chi_1, \chi_2]$
remains constant when 
$\chi_2$ is a fixed character of $\pi_0(Z(\varphi_2))$
but
$\chi_1$ varies among characters of $\pi_0(Z(\varphi_1))$.
This is achieved by combining
Corollary \ref{invariance} with Lemma \ref{components}.
We begin with the following lemma whose straightforward proof
will be omitted.
      
\begin{lem}
Let $N$ be a normal subgroup of a finite group $G$ with $A=G/N$
an abelian group.
Let $\pi$ be an irreducible representation of $N$.
Then any two irreducible representations $\pi_1$ and $\pi_2$ of $G$
containing $\pi$ on restriction
to $N$ are twists of each other by characters of $G/N$, i.e.,
$$
\pi_2 \cong \pi_1 \otimes \chi,
$$
for $\map{\chi}{G/N}{\mathbb C\mult}$.
\end{lem}

\begin{cor} \label{invariance}
If $N$ is a normal subgroup of a group $G$
with $A=G/N$ a finite abelian group,
and $\pi$ an irreducible representation of $N$,
then all irreducible $G$-submodules of $\Ind_N^G(\pi)$
appear in it with the same multiplicity.
\end{cor}

\begin{lem} \label{components}
Let $\G_1$ and $\G_2$ be two connected reductive groups over a local field $k$,
and let
$\map{\grouphom}{\G_1}{\G_2}$ be a $k$-homomorphism
that is a central isogeny when restricted to their derived subgroups,
and giving rise to a homomorphism 
$\map{\lsup L\grouphom}{\lsup L \G_2}{\lsup L\G_1}$
of the $L$-groups. 
Let $\varphi_2: W'_k \rightarrow {}^L\G_2$,
and
$\varphi_1= {}^L\grouphom \circ \varphi_2 : W'_k \rightarrow {}^L\G_1$
be associated Langlands parameters. 
Then for the associated homomorphism of finite groups
$\map{\grouphom^\star}%
{\pi_0(Z_{\widehat{\G}_2}(\varphi_2))}{\pi_0(Z_{{\widehat {\G}}_1}(\varphi_1))}$,
 the image is normal with abelian cokernel.
\end{lem}

\begin{proof}
It suffices to prove the lemma separately in the two cases:
\begin{enumerate}
\item
$\map{\grouphom}{\G_1}{\G_2}$ is injective as a homomorphism of algebraic groups.
\item
$\map{\grouphom}{\G_1}{\G_2}$ is surjective as a homomorphism of algebraic groups.
\end{enumerate}
We will do only the first case, the other being very similar.

Assume then that
$\map{\grouphom}{\G_1}{\G_2}$ is injective, and thus
$\map{\widehat{\grouphom}}{\widehat{\G}_2}{\widehat{\G}_1}$
is surjective with kernel say
$\widehat{Z}$.
Use $\map{\varphi_2}{W'_k}{\lsup L\G_2}$
and
$\varphi_1= \lsup L\grouphom \circ \varphi_2:W'_k \longrightarrow \lsup L\G_1$
to give
$\widehat{\G}_2$ and $\widehat{\G}_1$ a $W'_k$-group structure
such that we have an exact sequence of $W'_k$-groups:
$$
1
\longrightarrow \widehat{Z}
\longrightarrow \widehat{\G}_2
\longrightarrow \widehat{\G}_1
\longrightarrow 1.
$$
This gives rise to a long exact sequence of $W'_k$-cohomology sets:
$$
1
\longrightarrow \widehat{Z}^{W'_k}
\longrightarrow \widehat{\G}^{W'_k}_2
\longrightarrow \widehat{\G}^{W'_k}_1
\longrightarrow H^1(W'_k, \widehat{Z})
\longrightarrow \cdots.
$$
Equivalently, we have the exact sequence of groups:
$$
1
\longrightarrow Z_{\widehat{\G}_2}(\varphi_2) /\widehat{Z}^{W'_k}
\longrightarrow Z_{\widehat{\G}_1}(\varphi_1)
\longrightarrow A
\longrightarrow 1,
$$
where $A \subset H^1(W'_k, \widehat{Z}) $, a locally compact abelian group.
Taking $\pi_0$ of the terms in the above exact sequence
which all fit together in a long exact sequence of
$\pi_i$'s (higher homotopy groups),
the assertion in the lemma follows on noting that if $E_1 \rightarrow E_2$
is a surjective map of locally compact and locally connected
topological groups, then the induced map $\pi_0(E_1) \rightarrow \pi_0(E_2)$
is also surjective.
\end{proof}

\section{Reduction of the conjecture to the case of tempered representations}
\label{sec:tempered}
As before,
let $\G_1$ and $\G_2$ be two reductive groups over a local field $k$,
and let $\grouphom: \G_1\rightarrow \G_2$ be a $k$-homomorphism
that is a central isogeny when restricted to their derived subgroups,
giving rise to  the restriction functor
$\grouphom^\star : \RRf(\G_2(k)) \rightarrow \RRf(\G_1(k))$.

\begin{lem}
\label{lem:max-quotient}
Let $V$ be a finite-length representation of $\G_2(k)$
with maximal semi-simple quotient $Q$.
Then $\grouphom^\star Q$ is
the maximal semi-simple quotient of $\grouphom^\star V$, a finite-length 
representation of $G_1(k)$.
\end{lem}

\begin{proof}
It suffices to observe that a finite-length representation of $\G_2(k)$
is semisimple if and only if its image under $\grouphom^\star$
is a finite-length, semi-simple representation of  $\G_1(k)$.
If $\Z(\G_1)(k)\cdot\G_1(k)$
is of finite index in $\G_2(k)$, such as when $k$ is of characteristic
zero, then this is easy to see.
By a
theorem of Silberger \cite{silberger:isogeny-restriction},
irreducible representations of $\G_2(k)$ remain
finite-length semi-simple representations when restricted to $\G_1(k)$,  
and the lemma follows in general.
\end{proof}

To set up the next result,
let $P_2=M_2N_2$ be a Levi factorization of a parabolic subgroup
in $\G_2$.
If we let 
$P_1=\grouphom\inv(P_2)$,
$M_1=\grouphom\inv(M_2)$,
and
$N_1=\grouphom\inv(N_2)$,
then
$P_1=M_1N_1$ is a Levi factorization of a parabolic subgroup
in $\G_1$.
Then $\map{\grouphom}{M_1}{M_2}$ gives us a restriction functor
$\RRf(M_2(k)) \rightarrow \RRf(M_1(k))$
that we will also denote by $\grouphom^\star$.
Since $\grouphom$ gives an isomorphism
$\abmap{\G_1(k)/P_1(k)}{\G_2(k)/P_2(k)}$,
we have  the following commutative diagram:
\[
\begin{xy}
\xymatrix{
\RRf(\G_2(k)) \ar[r]^{\grouphom^\star} & \RRf(\G_1(k)) \\
\RRf(M_2(k)) \ar[r]^{\grouphom^\star}
\ar[u]^{\Ind_{P_2(k)}^{\G_2(k)}}
&
\RRf(M_1(k))
\ar[u]^{\Ind_{P_1(k)}^{\G_1(k)}}
\\
}
\end{xy}
\]

\begin{lem}
\label{lem:langlands-quotient}
Let $\sigma_2$ be an irreducible,
essentially tempered representation of $M_2(k)$
with strictly positive exponents along the center $\Z(M_2)(k)$ of $M_2(k)$.
Write
$$
\grouphom^\star \sigma_2
=
\sum_\alpha m_\alpha \sigma_{1,\alpha},
$$
a sum of irreducible, essentially tempered representations of $M_1(k)$
with (finite) multiplicities $m_\alpha$.
Let $\pi_2$ be the Langlands quotient of the standard module
$\Ind_{P_2(k)}^{\G_2(k)} \sigma_2$,
and $\pi_{1,\alpha}$ the Langlands quotients of
$\Ind_{P_1(k)}^{\G_1(k)} \sigma_{1,\alpha}$.
Then
$$
\grouphom^\star \pi_2 = \sum_\alpha m_\alpha \pi_{1,\alpha}.
$$
\end{lem}

\begin{proof}
Clearly,
$$
\grouphom^\star
\Ind_{P_2(k)}^{\G_2(k)} \sigma_2
=
\Ind_{P_1(k)}^{\G_1(k)} \grouphom^\star \sigma_2
=
\sum_\alpha m_\alpha \Ind_{P_1(k)}^{\G_1(k)} \sigma_{1,\alpha}.
$$
Since ``taking maximal semi-simple quotient''
commutes with direct sum, our result follows from
Lemma \ref{lem:max-quotient}.
\end{proof}

\begin{cor}
If Conjecture \ref{conj:main} is true for tempered representations,
then it is true in general.
\end{cor}

\begin{proof}
Every representation $\pi_2$ of $\G_2(k)$
can be realized as a Langlands quotient of a standard
module $\Ind_{P_2(k)}^{\G_2(k)}\sigma_2$ for an essentially tempered
representation $\sigma_2$ of $M_2(k)$.
The Langlands parameter
$\map{\varphi_2}{W_F'}{\lsup L \G_2}$ for $\pi_2$
is the same as the Langlands parameter $\varphi_2$ for $\sigma_2$
considered as a map
$W_F'
\stackrel{\varphi_2}{\longrightarrow} \lsup L M_2
\longrightarrow \lsup L \G_2$.
The component groups of these parameters, and thus the representations
of these component groups,
correspond as discussed in \cite{dprasad:MVW-involution}*{\S5}.
Therefore, our result is a consequence of
Lemma \ref{lem:langlands-quotient}.
\end{proof}

\section{Consequences of the conjecture}
\label{sec:consequences}

If the group of connected components $\pi_0(Z_{\widehat\G_1}(\varphi_1))$
is known to be abelian,
as is the case when $\G_1$ is any of the groups
$\SL_n$, $\U_n$, $\SO_n$, and $\Sp_n$,
then our conjecture predicts that for any homomorphism
$\map{\grouphom}{\G_1}{\G_2}$ of connected reductive algebraic
groups that is an isomorphism up to center
(i.e., $\map{\bar\grouphom}{\G_1/\Z_1}{\G_2/\Z_2}$ is an isomorphism of algebraic groups, where
$\Z_i$ is the center of $\G_i$), any irreducible representation of $\G_2(k)$ when restricted
via $\grouphom$ to $\G_1(k)$ decomposes as a sum of irreducible representations
of $\G_1(k)$ with multiplicity $\leq 1$.

We note that 
by our earlier work \cite{adler-dprasad:mult-one},
we know that multiplicity is $\leq 1$ whenever the pair
$(\G_1,\G_2)$ is
$(\SL_n,\GL_n)$,
or (when the characteristic of $k$ is not two)
either
$(\OO_n, \GO_n)$
or
$(\Sp_n, \GSp_n)$.
In the next section, we will see that multiplicity $\leq 1$ also
holds for $(\U_n, \GU_n)$.
Gee and Ta\"{\i}bi
\cite{gee-taibi:arthurs-multiplicity-gsp4}*{Proposition 8.2.1}
show that multiplicity $\leq 1$ holds for the pair $(\SO_n,\GSO_n)$
if $k$ has characteristic zero.

\section{Generalities on restriction to unitary and special unitary groups}
\label{sec:unitary}
Let $E/k$ denote a separable quadratic extension of nonarchimedean local fields,
$N= N_{E/k}$ the norm map from $E\mult$ to $k\mult$,
and $E_1$ the kernel of this map.

Let $B$ denote a nondegenerate $E/k$-hermitian form
on some $E$-vector space $V$ of some dimension $r$.
Then we can form algebraic groups $\SU(V,B)$, $\U(V,B)$,
and $\GU(V,B)$ whose $k$-points
consist respectively of the elements of $\SL(r,E)$ that preserve $B$;
the elements of $\GL(r,E)$ that preserve $B$;
and
the elements of $\GL(r,E)$
that preserve $B$ up to a scalar in $k\mult$.
The group $\GU(V,B)$ comes equipped with a map
$\map{\simil }{\GU(V,B)}{\GL_1}$
called the \emph{similitude character}.
We will write our algebraic groups as
$\SU(r)$, $\U(r)$, and $\GU(r)$ when $V$ and $B$ are understood.

If $G$ is a group, $H$ is a subgroup, and $G/Z(G)H$ is cyclic, then
every irreducible representation of $G$ restricts to $H$ without multiplicity.
How far can we exploit this fact?

\begin{thm}
\label{thm:elementary}
Let $p$ be the residual characteristic of $k$.
\begin{enumerate}[(a)]
\item
\label{item:GU-U}
All irreducible representations of $\GU(r)(k)$
decompose without multiplicity upon restriction to $\U(r)(k)$.
Such a restriction is irreducible when $r$ is odd,
and has at most two components when $r$ is even.
\item
All irreducible representations of $\U(r)(k)$ 
decompose without multiplicity upon restriction to $\SU(r)(k)$
when $r$ is coprime to $p$,
or $k=\Q_p$ ($p$ odd).
\item
All irreducible representations of $\GU(r)(k)$
decompose without multiplicity upon restriction to $\SU(r)(k)$
when $r$ is odd and coprime to $p$.
\end{enumerate}
\end{thm}

\begin{proof}
\begin{enumerate}[(a)]
\item
Let $\map{\simil }{\GU(r)}{\GL(1)}$ denote the similitude character.
Clearly the group $\GU(r)$ contains the scalar matrices $eI_r$ for all
$e\in E\mult$,
and for such matrices the similitude is
$N_{E/k}(e)$.
Therefore,
the image under $\simil $ of the center of $\GU(r)(k)$
is
$N_{E/k}(E\mult)$,
so $\simil $ thus gives an isomorphism
$$
\frac{\GU(r)}{Z(\GU(r))\U(r)}
\stackrel{\sim}{\longrightarrow}
\frac{\Image(\simil )}{N(E\mult)}.
$$
A scalar $a\in k\mult$
is a similitude for some linear transformation $g$ of $V$
if and only if for all $v,w \in V$, we have that $B(gv,gw) = a\cdot B(v,w)$.
That is, $B$ and $a\cdot B$ are equivalent Hermitian forms.
It is known that two Hermitian forms
over a non-archimedean local field $k$
are equivalent if and only if their discriminants,
which are elements of $k\mult/ N(E\mult)$,
are the same.
Therefore,
$B$ and $aB$ are equivalent if and only if
$\disc B = a^r \disc B$ in $k\mult/N(E\mult)\cong \ZZ/2$.
Thus, if $r$ is even, then $B$ and $aB$ are equivalent
for $a$ an arbitrary element of $k\mult$,
but if $r$ is odd, then $a$ must lie in $N(E\mult)$.
Thus,
$$
\frac{\GU(r)}{\Z(\GU(r)) \U(r)}
\cong
\ZZ/2
\quad \text{or} \quad
\{1\}.
$$
\item
Let $R_E$ and $P_E$ denote the ring of integers and
prime ideal for $E$.
The determinant character gives us an isomorphism
$$
\det \colon \frac{\U(r)(k)}{Z(\U(r))(k)\SU(r)(k)}
\stackrel{\sim}{\longrightarrow}
\frac{E_1}{(E_1)^r}.
$$
As an abstract group, $E_1$ inherits a direct product decomposition
from $R_E\mult \cong k_E\mult \times (1+P_E)$.
Thus, $E_1$ is a direct product of a cyclic group (of order coprime to $p$)
and a pro-$p$-group $A$,
implying that $E_1 / E_1^r$ is cyclic if and only $A/A^r$ is cyclic.
But this latter quotient is trivial if $r$ is coprime to $p$,
and is cyclic if $k=\Q_p$ ($p$ odd).
\item
This follows from the previous two parts of the theorem.
\qedhere
\end{enumerate}
\end{proof}

\section{An example of multiplicity upon restriction}
\label{sec:example-mult}
Let $\varpi$ be a uniformizer of $k$, $E/k$ an unramified quadratic
extension, $R_k$ and $R_E$ the rings of integers in $k$ and $E$,
and $\ff$ and $\ff_E$ the residue fields.
Let $V$ be a $4d$-dimensional hermitian space over $E$,
with hyperbolic basis
$\{e_1, f_1, \dots, e_{2d}, f_{2d}\}$.
Thus, $\langle e_i,f_i\rangle = 1$ for all $1\leq i \leq 2d$, and all the other products being $0$.
Let $\U(V)$ be the corresponding unitary group.
Define the lattice $\LL$ in $E$ by
$$
\LL =
\Span_{R_E}\{
	e_1,f_1,\ldots,e_d,f_d,
	\varpi e_{d+1},f_{d+1},\ldots, \varpi e_{2d}, f_{2d}
	\}.
$$
Clearly,
$\LL^\vee
:=
\{v\in V | \text{$\langle v,\ell\rangle\in R_E$ for all $\ell\in\LL$}\}$
is given by
$$
\LL^\vee =
\Span_{R_E}\{
	e_1,f_1,\ldots,e_d,f_d,
	e_{d+1},\varpi\inv f_{d+1},\ldots, e_{2d}, \varpi\inv f_{2d}
	\}.
$$
Observe that
$$
\varpi \LL^\vee \subseteq \LL \subseteq \LL^\vee,
$$
and $\LL^\vee/\LL$ and $\LL/\varpi\LL^\vee$
are $2d$-dimensional hermitian spaces over $\ff_E$
with natural hermitian structures.
For example, given two elements $\ell_1$ and $\ell_2$ in $\LL^\vee$
with images 
$\overline\ell_1$ and $\overline\ell_2$ in $\LL^\vee/\LL$, the hermitian
structure on $\LL^\vee/\LL$ is defined by 
having $\langle \overline\ell_1,\overline\ell_2\rangle$ as
the image of 
$\varpi\langle\ell_1,\ell_2\rangle$
(which belongs to $R_E$)
in $\ff_E$.

Define $K=\U(\LL)$ to be the stabilizer of the lattice $\LL$ in $\U(V)$,
i.e.,
$\U(\LL) =  \{ g \in \U(V) | \text{$g\ell \in \LL$ for all $\ell\in\LL$} \}$.
If an element of $\U(V)$ preserves $\LL$, then it clearly preserves
$\LL^\vee$ and $\varpi\LL$, giving a map
$\abmap{\U(\LL)}{\U(2d,\ff)\times \U(2d,\ff)}$.
Similarly, we have a map
$\abmap{\SU(\LL)}{\Special(\U(2d)\times \U(2d))(\ff)}$.

Let $g_0\in \GU(V)$ be defined by 
(for $i \leq d$)
$$
e_i \mapsto e_{d+i},
\quad
f_i \mapsto \varpi\inv f_{d+i},
\quad
e_{d+i} \mapsto \varpi\inv e_i,
\quad
f_{d+i} \mapsto f_i,
$$
Clearly,
$g_0$ has similitude factor $\varpi\inv$,
and $g_0\LL = \LL^\vee$.
Therefore, we have
$$
g_0\U(\LL)g_0\inv =\U(\LL^\vee).
$$
Thus conjugation by $g_0$ induces an isomorphism of $\U(\LL)$
into $\U(\LL^\vee)$, making the following diagram commute:
$$
\begin{xy}
\xymatrix{
\U(\LL) \ar[r]^{g_0} \ar[d] & \U(\LL^\vee) \ar[d] \\
\U(2d,\ff)\times\U(2d,\ff) \ar[r]^j &
	\U(2d,\ff)\times\U(2d,\ff) 
}
\end{xy}
$$
where $j(x,y) = (y,x)$.

\begin{thm}
\label{thm:mult-two}
Let $\rho$ be any irreducible cuspidal representation of $\U(2d)(\ff)$
such that $\rho\not\cong\rho\chi$, where $\chi$ is a quadratic character
of $\U(2d)(\ff)$ trivial on $\SU(2d)(\ff)$.
Let $\sigma := \infl(\rho\otimes\rho\chi)$
denote the inflation of $\rho\otimes\rho\chi$
from $(\U(2d)\times\U(2d))(\ff)$ to $\U(\LL)$
and let $\pi = \cInd_{\U(\LL)}^{\U(V)} \sigma$.
Then $\pi\oplus \pi^{g_0}$ extends to an irreducible representation
$\widetilde\pi$ of $\GU(V)$ whose restriction to $\SU(V)$
decomposes with multiplicity two.
\end{thm}

\begin{proof}
From Moy-Prasad \cite{moy-prasad:jacquet}*{Proposition 6.6},
$\pi$ is an irreducible, supercuspidal representation of $\U(V)$.
Let $\pi$ also denote one of its extensions to
$Z(\GU(V)) \U(V)$.
From the last sentence of \cite{moy-prasad:k-types}*{Theorem 5.2},
$\pi^{g_0} \not\cong\pi$, so the sum $\pi\oplus \pi^{g_0}$
extends to an irreducible (also supercuspidal) representation
$\widetilde\pi$
of $\GU(V)$.
By the induction-restriction formula (observe that by the explicit description of $\U(\LL)$, $\det: \U(\LL)
\rightarrow E_1$ is surjective, and hence $\U(\LL)\SU(V)=\U(V)$),
\begin{align*}
\pi|_{\SU(V)}
& = \cInd_{\SU(\LL)}^{\SU(V)}
	(\sigma|_{\SU(\LL)}), \\
\pi^{g_0}|_{\SU(V)}
& = \cInd_{\SU(\LL)}^{\SU(V)}
	(\sigma^{g_0}|_{\SU(\LL)}) .
\end{align*}
Since $\rho\otimes\rho\chi \cong \rho\chi\otimes \rho$ as representations of
$\Special(\U(2d)\times\U(2d))(\ff)$,
we have
that $\sigma\cong\sigma^{g_0}$ as representations of
$\SU(\LL)$, so
\[
\widetilde\pi|_{\SU(V)}
= (\pi\oplus \pi^{g_0})|_{\SU(V)} 
= 2 \cdot \cInd_{\SU(\LL)}^{\SU(V)}
	(\sigma|_{\SU(\LL)}).
\qedhere
\]
\end{proof}

In order to have an example of multiplicity at least two,
it is thus sufficient to find a representation $\rho$ of $\U(2d)(\ff)$
such that $\rho\not\cong\rho\chi$,
as in the theorem.
In fact, most irreducible Deligne-Lusztig
cuspidal representations of $\U(2d)(\ff)$ will have this property,
as they restrict irreducibly to $\SU(2d)(\ff)$.

\begin{rem}
In a future work, we will expand upon the example given in the Theorem,
whose essence is the following.
Given a supercuspidal representation of $\G_2(k)$
whose restriction to $\G_1(k)$ has regular components
(in the sense of 
Kaletha \cite{kaletha:regular-sc}),
then the components occur with multiplicity one.
(Nevins
\cite{nevins:restricting-toral-supercuspidals}
already verified this for many cases.)
If the components are not regular, then higher
multiplicities can occur.

Our example  begins with $\rho$, an irreducible cuspidal representation
of $\U(2d)(\ff)$
that arises via Deligne-Lusztig induction from a character
$\theta$ of the group of $\ff$-points of an anisotropic torus
$\sfT \subset\U(2d)$.
Suppose also that the restriction of $\theta$ to $\sfT(\ff)\cap \SU(2d)(\ff)$
remains regular
so that the restriction of $\rho$ to $\SU(2d)(\ff)$ remains irreducible.
The torus $\sfT\times \sfT \subset \U(2d)\times\U(2d)$
lifts to give an unramified torus $T\subset \GU(V)$,
and the character $\theta\otimes\theta\chi$ can be inflated
and extended to give a character $\Theta$ of $T$.
The representation $\widetilde\pi$ of $\GU(V)$
that we have constructed in the theorem is 
a regular supercuspidal representation in the sense of
Kaletha \cite{kaletha:regular-sc},
but the irreducible components of its restriction
to $\SU(V)$ are not since our
character $\Theta$ of $T$, when restricted to $T \cap \SU(V)$,
is not regular because of the presence
of the element $g_0\in\GU(V)$.

For depth-zero supercuspidal representations of quasi-split unitary
groups, the parahoric that we have used is the only one that can
lead to higher multiplicities.
\end{rem}

\section{Generalities on constructing higher multiplicities}
\label{sec:generalities}
In this section, we discuss some generalities underlying
the example of the previous section, which will be useful
for constructing higher multiplicities in general.

Let $G$ be a group, and $N$ a normal subgroup
of $G$ such that
$$
G/N \cong \ZZ/2 \oplus \ZZ/2.
$$
A good example to keep in mind is $G=Q_8=\{\pm 1, \pm i, \pm j, \pm k\}$,
the quaternion group of order $8$, and $N=\{\pm 1\}$.
Let $\omega_1$ and $\omega_2$ be two distinct, nontrivial
characters of $G$ that are trivial on $N$.

Suppose $\pi$ is an irreducible representation of $G$ such that
$$
\pi \cong \pi\otimes \omega_1 \cong \pi\otimes\omega_2.
$$
By \cite{gelbart-knapp:sl2}*{\S2},
$\pi|_N$
must be one of the following:
\begin{enumerate}
\item
a sum of four inequivalent, irreducible representations, or
\item
a sum of two copies of an irreducible representation.
\end{enumerate}
Deciding which of these two options we have is a subtle question,
and this is what we wish to do here.

Let $N_1 = \ker\{\map{\omega_1}{G}{\ZZ/2}\}$,
so that 
$G\supset N_1\supset N$.
Since $\pi \cong \pi\otimes \omega_1$,
$\pi|_{N_1} = \pi_1\oplus\pi_2$, a sum of inequivalent,
irreducible representations.
Further, since $\pi\cong\pi\otimes\omega_2$,
we have
$$
(\pi_1\oplus\pi_2) \cong (\pi_1\oplus\pi_2)\otimes \omega_{21},
$$
where $\omega_{21} = \omega_2|_{N_1}$,
a nontrivial character of $N_1$ of order 2.
Therefore, we have the following two possibilities
\begin{enumerate}[(i)]
\item
$\pi_1 \cong \pi_1\otimes \omega_{21}$,
\item
$\pi_2 \cong \pi_1\otimes \omega_{21}$,
\end{enumerate}
In case (i), $\pi_1$, which is an irreducible representation of $N_1$,
decomposes when restricted to $N$ into
two inequivalent irreducible representations,
and therefore $\pi$ has
at least two inequivalent irreducible subrepresentations
when restricted to $N$, hence in case (i),
$$
\pi|_N = \text{a sum of $4$ inequivalent, irreducible representations}.
$$
In case (ii), clearly $\pi|_N$ is twice an irreducible representation.

How does one then construct an example of an irreducible representation
$\pi$ of $G$ for which
$\pi|_{N}$ is twice an irreducible representation?
We start with an irreducible representation $\pi_1$ of $N_1$
such that the following equivalent conditions hold:
\begin{enumerate}[(i)]
\item
$\pi_1$ does not extend to a representation of $G$;
\item
$\pi_1^g \not\cong \pi_1$ for some $g\in G$.
\end{enumerate}
Given such a representation $\pi_1$ of $N_1$,
next we must ensure that
$$
\pi_1^g \cong \pi_1\otimes \omega_{21} \text{ for } g \in G \setminus N.
$$
If we understand $N_1$,
together with the action of $G$ on
the representations of $N_1$,
then the condition
$$
\pi_1^g \cong \pi_1\otimes \omega_{21} \not\cong \pi_1
$$
is checkable,
constructing an irreducible representation
$\pi = \Ind_{N_1}^G \pi_1$ of $G$ such that
$$
\pi|_N= 2 \pi_1|_N.
$$

\textbf{Acknowledgements.}
As mentioned in the introduction,
most of this work is several years old.
In the mean time, there is the recent work of
K. Choiy \cite{choiy:multiplicity-restriction}
formulating the conjectural multiplicity formula
quite independently of us,
and 
proving it for tempered representations
under the assumption that the Langlands correspondence
has many of the expected properties.

For their hospitality,
we thank the CIRM (Luminy),
and the first-named author also thanks
the Tata Institute of Fundamental Research (Mumbai). We thank Hengfei Lu for his comments on this work.
The final writing of this work was supported by a grant of
the Government of the Russian Federation
for the state support of scientific research carried out
under the supervision of leading scientists,
agreement 14.W03.31.0030 dated 15.02.2018.

\end{document}